\documentclass[12pt]{amsart}

\usepackage{times,amssymb,amsmath,hyperref}
\usepackage{amscd,amsfonts,amsmath,amssymb,epsf,rlepsf,amsthm,amstext,latexsym}
\usepackage{enumerate,pifont,graphicx}
\usepackage{color}
\usepackage[english]{babel}
\usepackage[T1]{fontenc}
\usepackage{epsfig}
\usepackage[latin1]{inputenc}

\usepackage{enumerate,latexsym}
\usepackage{latexsym}
\usepackage{amsmath,amssymb}% CUIDADO: Si uso estos packages, $\widetilde{}$ no funciona a veces
\usepackage{graphicx}

\newcommand{\Int}{\mbox{\rm Int}}

\newtheorem{theorem}{Theorem}[section]
\newtheorem{lemma}[theorem]{Lemma}
\newtheorem{conjecture}[theorem]{Conjecture}

\newtheorem{claim}[theorem]{Claim}
\newtheorem{proposition}[theorem]{Proposition}
\newtheorem{remark}[theorem]{Remark}
\theoremstyle{definition}
\newtheorem{definition}{Definition}[section]

\theoremstyle{remark}

\def\square{\hfill${\vcenter{\vbox{\hrule height.4pt \hbox{\vrule
width.4pt height7pt \kern7pt \vrule width.4pt} \hrule height.4pt}}}$}

\newcommand{\SI}{\partial_\infty {\Bbb H}^2\times {\Bbb R}}

\newcommand{\si}{S^1_{\infty}}
\newcommand{\PI}{\partial_{\infty}}

\newcommand{\BH}{\Bbb H}
\newcommand{\BHH}{{\Bbb H}^2\times {\Bbb R}}
\newcommand{\BR}{\Bbb R}

\newcommand{\ff}{\frac{1}{2}}
\newcommand{\p}{\mathcal{P}}
\newcommand{\U}{\mathcal{U}}
\newcommand{\V}{\mathcal{V}}
\newcommand{\I}{\mathcal{I}}

\newcommand{\ov}{\overline}
\newcommand{\cF}{\mathcal{F}}

\newcommand{\cP}{\mathcal{P}}

\newcommand{\BZ}{\mathbb{Z}}

\newcommand{\ve}{\varepsilon}

\newcommand{\A}{\mathcal{A}}
\newcommand{\B}{\mathcal{B}}
\newcommand{\cL}{\mathcal{L}}

\newcommand{\wh}{\widehat}
\newcommand{\wt}{\widetilde}
\newcommand{\C}{\mathcal{C}}
\newcommand{\cC}{\mathcal{C}}

\newcommand{\N}{\mathbb{N}}

\newcommand{\R}{\mathbb{R}}
\newcommand{\rth}{\R^3}
\newcommand{\HH}{\mathbf{H}}
\newcommand{\ben}{\begin{enumerate}}
\newcommand{\bit}{\begin{itemize}}
\newcommand{\een}{\end{enumerate}}
\newcommand{\eit}{\end{itemize}}
\newcommand{\G}{\Gamma}
\newcommand{\ed}{\end{document}}

\usepackage{times}

\begin{document}
%\today
\title{Non-properly Embedded $H$-Planes in $\BHH$}
\author{Baris Coskunuzer}
\thanks{The first author is partially supported by BAGEP award of the Science Academy, and a Royal Society Newton Mobility Grant.}
\author{William H. Meeks III}
\thanks{The second author was supported in part by NSF Grant DMS -
   1309236. Any opinions, findings, and conclusions or recommendations
   expressed in this publication are those of the authors and do not
   necessarily reflect the views of the NSF}
\author{Giuseppe Tinaglia}
\thanks{The third author was partially supported by EPSRC grant no. EP/M024512/1, and a Royal Society Newton Mobility Grant.}
\address{Department of Mathematics \\ Boston College \\ Chestnut Hill, MA 02467}
\email{coskunuz@bc.edu}
\address{Department of Mathematics \\ University of Massachusetts \\ Amherst, MA 01002}
\email{profmeeks@gmail.com}
\address{Department of Mathematics \\ King's College London, London}
\email{giuseppe.tinaglia@kcl.ac.uk}

\begin{abstract}
For any $H \in (0,\frac{1}{2})$, we construct  complete, non-proper,
stable, simply-connected surfaces  embedded in $\BHH$ with constant mean curvature $H$.
\end{abstract}


\maketitle

\section{Introduction}
In their ground breaking work~\cite{cm35}, Colding and
Minicozzi proved that  complete minimal surfaces  embedded
in $ \rth$ with finite topology are proper.
 Based on the techniques in~\cite{cm35},  Meeks
and Rosenberg~\cite{mr13} then proved that
complete minimal surfaces with positive
injectivity embedded in $\mathbb{R}^3$ are proper.
More recently,  Meeks and Tinaglia~\cite{mt15} proved that  complete
 constant mean curvature surfaces  embedded in
$\rth$  are proper if they
have finite topology or have positive injectivity radius.

%With the convention that the mean curvature function of an oriented  surface
%is the pointwise average of its principal curvatures,
%these results of Meeks and Tinaglia in $\rth$ should generalize
%to show that a complete embedded
%surface $\Sigma$ of constant mean curvature $H\in [1,\infty)$
%in a complete hyperbolic
%three-manifold is proper if $\Sigma$ has finite topology
%or it is connected and has positive injectivity
%radius; this is work in progress in~\cite{mt11}.

In contrast to the above  results, in this paper we prove the
following existence theorem for non-proper, complete,
simply-connected surfaces embedded in $\BHH$
with constant mean curvature $H\in (0,1/2)$. The convention used here
is that the mean curvature function of an oriented surface $M$ in an oriented
Riemannian three-manifold $N$
is the pointwise average of its principal curvatures.

The catenoids in $\BHH$ mentioned in the next theorem
are defined at the beginning of Section~\ref{Hcatenoid}.

\begin{theorem}\label{main} For any $H\in (0,1/2)$ there
exists a complete, stable, simply-connected surface $\Sigma_H$ embedded
in $\BHH$ with constant mean curvature $H$
satisfying the following properties:
\begin{enumerate}
\item The closure of $\Sigma_H$ is a lamination with three
leaves, $\Sigma_H$, $C_1$ and $C_2$,
where $C_1$ and  $C_2$
are stable catenoids of constant mean curvature $H$ in
$\mathbb{H}^3$ with the same axis of revolution $L$.
In particular, $\Sigma_{H}$ is not properly embedded in $\BHH$.
%\item The asymptotic boundary of $\Sigma_H$ is a pair of
%embedded curves  in $\partial_{\infty} \BHH$ which spiral
%into the union  of the round circles which are the
%asymptotic boundaries of $C_1$ and $C_2$.
\item Let $K_L$ denote the  Killing field generated by
rotations around $L$. Every
integral curve of $K_L$ that lies in the region between
$C_1$ and $C_2$ intersects $\Sigma_H$
transversely in a single point.  In particular,
the closed region between
$C_1$ and $C_2$ is foliated by  surfaces   of constant mean curvature $H$,
where the leaves are
$C_1$ and $C_2$ and  the rotated images $\Sigma_H ({\theta})$ of
$\Sigma$ around $L$ by angle $\theta\in [0,2\pi)$.
\end{enumerate}
\end{theorem}

When $H=0$,
Rodr\'{\i}guez and Tinaglia~\cite{rodt1} constructed non-proper, complete
minimal planes embedded in $\BHH$.
However, their construction does  not generalize to
produce complete, non-proper planes embedded in $\BHH$
with non-zero constant mean curvature.
Instead, the construction presented in this paper is related to
the techniques developed by the authors in~\cite{cmt1}  to obtain
examples of  non-proper, stable, complete planes
embedded in $\mathbb H^3$ with constant mean curvature $H$, for any $H\in[0,1)$.

There is a general conjecture related to Theorem~\ref{main}
and the previously stated positive properness results.
Given $X$ a Riemannian three-manifold, let $\mbox{\rm Ch}(X):= \inf_{S\in\mathcal{S}} \frac{\text{Area}(\partial S)}{\text{Volume}(S)},$
where $\mathcal{S}$ is the set of all smooth compact domains in $X$. Note that when the volume of $X$ is infinite, $\mbox{\rm Ch}(X)$ is the Cheeger constant.

\begin{conjecture} \label{conj1.2} Let $X$ be a simply-connected,
homogeneous three-manifold.  Then for
any $H\geq\frac12\mbox{\rm Ch}(X)$, every  complete, connected
$H$-surface embedded in $X$ with positive injectivity
radius or  finite topology  is proper. On the other hand,
if $\mbox{\rm Ch}(X)>0$, then there exist non-proper complete
$H$-planes in $X$ for every $H\in [0,\frac12 \mbox{\rm Ch}(X))$.
\end{conjecture}

By the work in~\cite{cm35},
Conjecture~\ref{conj1.2} holds for $X=\rth$ and it  holds in $\mathbb{H}^3$ by work in progress
in~\cite{mt11}. Since the Cheeger constant of
$\BHH$ is $1$, Conjecture~\ref{conj1.2} would imply that Theorem~\ref{main}
 (together with the existence of complete non-proper minimal
planes embedded in $\BHH$ found in~\cite{rodt1}) is a sharp result.

\section{Preliminaries} \label{prelim}
In this section, we will review the basic properties of $H$-surfaces,
a concept that we next define. We will
call a smooth oriented surface $\Sigma_H$ in $\BHH$ an {\em $H$-surface} if it is
embedded and its mean curvature is constant
equal to $H$; we will assume that $\Sigma_H$ is appropriately
oriented so that $H$ is non-negative. We will  use the cylinder model
of $\BHH$ with coordinates $(\rho, \theta, t)$; here $\rho$ is the hyperbolic distance
from the origin (a chosen base point) in $\BH^2_0$, where $\BH^2_t$ denotes $\BH^2\times\{t\}$.
  We next describe the $H$-catenoids mentioned in the Introduction.

The following $H$-catenoids family will play a particularly important role in our construction.

\subsection{Rotationally invariant vertical $H$-catenoids $\C^H_d$} \label{Hcatenoid} \

We begin this section by recalling several results in~\cite{nsest1, ner1}.
Given $H\in (0,\ff)$ and $d\in [-2H,\infty)$,
let
\[
\eta_d=\cosh^{-1}\left(\frac{2dH+\sqrt{1-4H^2+d^2}}{1-4H^2}\right)
\]
 and let
$\lambda_d\colon [\eta_d,\infty)\to [0,\infty)$ be the function defined as follows.
\begin{equation}\label{Hcat-eqn}
\lambda_d(\rho)= \int ^{\rho}_{\eta_d} \frac{d+2H\cosh r}{\sqrt{\sinh^2 r - ( d+2H\cosh r)^2}}dr.
\end{equation}
Note that $\lambda_d(\rho)$ is a strictly increasing function with
$\lim_{\rho\to\infty}\lambda_d(\rho)= \infty$ and derivative
$\lambda'_d(\eta_d)=\infty$ when $d\in (-2H,\infty)$.

In~\cite{nsest1} Nelli,   Sa Earp,   Santos and   Toubiana proved
that there exists a 1-parameter family of
embedded $H$-catenoids $\{\C^H_d \ | \ d\in (-2H,\infty)\}$ obtained by rotating
a generating curve $\lambda_d(\rho)$ about the $t$-axis. The generating
curve $\wh{\lambda}_d $   is obtained by doubling the curve
$(\rho, 0, \lambda_d(\rho))$, $\rho\in[\eta_d,\infty)$, with its reflection
$(\rho, 0, -\lambda_d(\rho))$, $\rho\in[\eta_d,\infty)$.
Note that $\wh{\lambda}_d$ is a smooth curve and that the necksize, $\eta_d$,
is a strictly increasing  function in $d$ satisfying the properties that
$\eta_{-2H}=0$   and $\lim_{d\to \infty}\eta_d=\infty$.

If $d=-2H$, then by rotating the curve $(\rho, 0, \lambda_d(\rho))$ around the $t$-axis
one obtains a simply-connected $H$-surface $E_H$ that is an entire graph over $\BH^2_0$.
We denote by $-E_H$ the reflection of $E_H$ across $\BH^2_0$.

We next recall the definition of the mean curvature vector.

\begin{definition}\label{mcvector}
Let  $M$ be an oriented surface in an oriented Riemannian three-manifold and suppose that
$M$ has non-zero mean curvature $H(p)$ at $p$. The
{\bf mean curvature vector at $p$} is $ \HH(p):=H(p)N(p)$, where $N(p)$ is its unit normal
vector at $p$. The mean curvature vector  $\HH(p)$ is independent of the orientation on $M$.
\end{definition}

Note that the mean curvature vector $\HH$ of $\C^H_d$
points into the connected component
of $\BHH -\C^H_d$ that contains the $t$-axis. The mean curvature vector of $E_H$ points
upward while the mean curvature vector of $-E_H$ points downward.

In order to construct the  examples described in Theorem~\ref{main},
we first obtain  certain geometric properties satisfied by
$H$-catenoids. For example, in the following lemma, we show that for certain
values of $d_1$ and $d_2$, the catenoids $\C^H_{d_1}$ and $\C^H_{d_2}$ are disjoint.

Given $d\in (-2H,\infty)$, let $b_{d}(t):=\lambda_d^{-1}(t)$ for $t\geq 0$; note that
$b_d(0)=\eta_{d}$. Abusing the notation let $b_d(t):=b_d(-t)$ for $t\leq 0$.

\begin{lemma}[Disjoint $H$-catenoids] \label{disjointlem}   Given $d_1>2$, there
exist $d_0>d_1$ and $\delta_0>0$ such that for any $d_2\in [d_0,\infty)$, then
\[
\inf_{t\in \mathbb R}( b_{d_2}(t)-b_{d_1}(t))\geq \delta_0.
\]
In particular, the corresponding $H$-catenoids are disjoint, i.e.  $\C^H_{d_1}\cap\C^H_{d_2}=\emptyset$.

Moreover, $b_{d_2}(t)-b_{d_1}(t)$ is decreasing for $t>0$ and increasing for $t<0$. In particular,
\[
\sup_{t\in \mathbb R}( b_{d_2}(t)-b_{d_1}(t))=b_{d_2}(0)-b_{d_1}(0)= \eta_{d_2}-\eta_{d_1}.
\]
\end{lemma}
The proof of the above lemma requires a rather lengthy computation that is given in the Appendix.

We next recall the well-known mean curvature comparison principle.

\begin{proposition}[Mean curvature comparison principle] \label{max}
Let $M_1$ and $M_2$ be two complete, connected embedded surfaces in a three-dimensional Riemannian
manifold.
Suppose that  $p\in M_1\cap M_2$  satisfies that
a neighborhood of $p$ in $M_1$ locally lies on the side of a neighborhood of $p$ in $M_2$ into which $\HH_2(p)$ is pointing. Then $|H_1|(p)\geq |H_2|(p)$. Furthermore, if $M_1$
and $M_2$ are constant mean curvature surfaces with $|H_1|=|H_2|$, then $M_1 =M_2$.
\end{proposition}

\section{The Examples}

For a fixed $H\in (0,1/2)$, the outline of  construction is as follows.
First, we will take two disjoint $H$-catenoids $\C_1$
and $\C_2$ whose existence  is given
in Lemma~\ref{disjointlem}. These catenoids $\C_1$, $ \C_2$ bound a region $\Omega$ in $\BHH$ with fundamental group
$\mathbb{Z}$. In the universal cover $\wt{\Omega}$ of
$\Omega$, we define a piecewise smooth  compact exhaustion
$\Delta_1\subset\Delta_2\subset...\subset\Delta_n\subset...$ of $\wt{\Omega}$.
Then, by solving the $H$-Plateau  problem for special curves $\Gamma_n\subset \partial \Delta_n$,
we obtain minimizing $H$-surfaces $\Sigma_n$ in $\Delta_n$ with $\partial \Sigma_n=\Gamma_n$.
In the limit set of these surfaces, we find  an $H$-plane $\p$ whose projection to $\Omega$
is the desired non-proper  $H$-plane $\Sigma_H\subset \mathbb{H}^2\times \R$.

\subsection{Construction of $\wt{\Omega}$} \

Fix $H\in(0,\ff)$ and $d_1,d_2\in (2,\infty)$, $d_1<d_2$, such that by Lemma~\ref{disjointlem},
the related $H$-catenoids $\C^H_{d_1}$ and $\C^H_{d_2}$ are disjoint;
note that in this case, $\C_{d_1}^H$ lies in the interior of the simply-connected component of
$\BHH-\C^H_{d_2}$. We will use the notation $\C_i:= \C^H_{d_i}$.
Recall that both catenoids have the same rotational axis, namely the $t$-axis, and recall that the mean
curvature vector $\HH_i$ of $\C_i$ points into the connected component of $\BHH -\C_i$ that contains
the $t$-axis. We emphasize here that $H$ is {\em fixed} and so we will omit describing it in
future notations.

Let $\Omega$ be the closed region in $\BHH$ between $\C_1$ and $\C_2$,
i.e., $\partial \Omega = \C_1\cup\C_2$ (Figure~\ref{univcover1}-left).
Notice that the set of boundary points at infinity
$\PI \Omega$ is equal to  $\si\times\{-\infty\} \cup \si\times\{\infty\}$, i.e., the corner circles
in $\SI$ in the product compactification, where we view $\mathbb{H}^2$ to be the open unit disk
$\{(x,y)\in \R^2 \mid x^2+y^2<1\}$ with base point the origin $\vec{0}$.

\begin{figure}[h]
\begin{center}
$\begin{array}{c@{\hspace{.2in}}c}

\relabelbox  {\epsfysize=2in \epsfbox{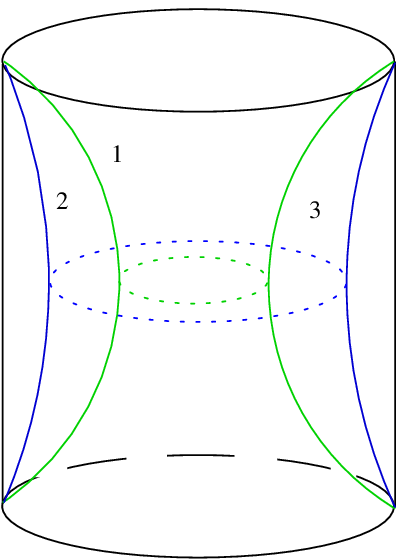}}
\relabel{1}{\footnotesize \color{green} $\C_1$} \relabel{2}{\footnotesize \color{blue}  $\C_2$}
\relabel{3}{\footnotesize $\Omega$}
\endrelabelbox &

\relabelbox  {\epsfxsize=2.5in \epsfbox{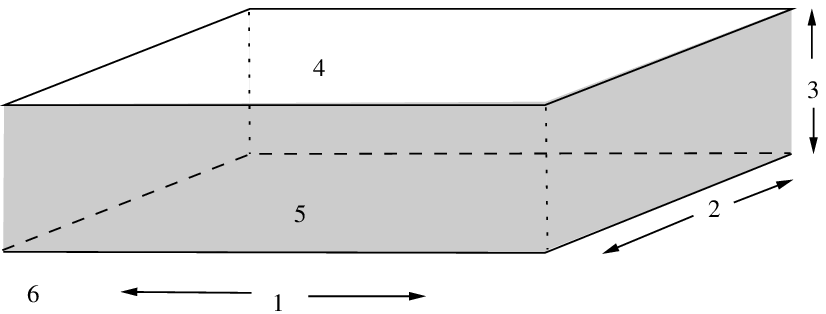}}
\relabel{1}{\footnotesize $\wt{\theta}$} \relabel{2}{\small $t$} \relabel{3}{\footnotesize
$\rho$} \relabel{4}{\footnotesize $\wt{\C}_2$} \relabel{5}{$\wt{\Omega}$} \relabel{6}{\footnotesize $\wt{\C}_1$}
\endrelabelbox \\ [0.4cm]
\end{array}$

\end{center}

\caption{ The induced coordinates $(\rho, \wt{\theta}, t)$ in $\wt{\Omega}$.} \label{univcover1}

\end{figure}

By construction, $\Omega$ is topologically a solid torus. Let $\wt{\Omega}$ be the universal
cover of $\Omega$. Then, $\partial \wt{\Omega} = \wt{\C}_1\cup\wt{\C}_2$ (Figure~\ref{univcover1}-right),
where $\wt{\C}_1,\wt{\C}_2$ are the respective lifts to $\wt{\Omega}$  of  ${\C}_1, {\C}_2$.
Notice that $\wt{\C}_1$ and $\wt{\C}_2$ are both $H$-planes, and the mean curvature vector $\HH$ points
outside of $\wt{\Omega}$ along $\wt{\C}_1$ while $\HH$ points inside of $\wt{\Omega}$ along $\wt{\C}_2$.
We will use the induced coordinates $(\rho, \wt{\theta}, t)$ on $\wt{\Omega}$ where
$\wt{\theta}\in (-\infty, \infty)$. In particular, if
\begin{equation}\label{covermap}
\Pi\colon \wt{\Omega}\to\Omega
\end{equation}
 is the covering map,
then $\Pi(\rho_o, \wt{\theta}_o, t_o)= (\rho_o, \theta_o, t_o)$ where $\theta_o\equiv\wt{\theta_o} \mod 2\pi$.

Recalling the definition of $b_i(t)$, $i=1,2$, note that a point $(\rho, \theta, t)$ belongs to $\Omega$
if and only if $\rho\in [b_1(t), b_2(t)]$ and we can write
\[
\wt{\Omega}= \{ (\rho, \wt{\theta}, t) \ | \ \rho\in[b_1(t), b_2(t)], \ \wt{\theta}\in \BR,\ t\in\BR\}.
\]

\subsection{Infinite Bumps in $\wt{\Omega}$.} \

Let $\gamma$ be the geodesic through the origin in $\BH^2_0$ obtained by intersecting $\BH^2_0$
with the vertical plane $\{\theta=0\} \cup \{\theta=\pi\}$.  For $s\in[0,\infty)$, let $\varphi_s$
be the orientation preserving hyperbolic isometry of $\BH^2_0$ that is the hyperbolic translation along
the geodesic $\gamma$
with $\varphi_s(0,0)=(s,0)$. Let
\begin{equation}\label{shifting}
\wh{\varphi}_s\colon \BHH\to \BHH, \quad
\wh{\varphi}_s(\rho,\theta,t)=(\varphi_s(\rho,\theta),t)
\end{equation}
be the related extended isometry of $\BHH$.

Let $\C_d$ be an embedded $H$-catenoid as defined in Section~\ref{Hcatenoid}.
Notice that the rotation
axis of the $H$-catenoid $\wh{\varphi}_{s_0}(\C_d)$ is the
vertical line $\{ (s_0, 0, t) \mid t\in \BR \}$.

Let $\delta:=\inf_{t\in \mathbb R}(b_2(t)-b_1(t))$, which gives an upper bound estimate for
the asymptotic distance between the catenoids; recall that
by our choices of $\C_1,\C_2$ given in Lemma~\ref{disjointlem}, we  have $\delta>0$.
Let $\delta_1=\frac{1}{2}\min \{\delta, \eta_1\}$ and let $\delta_2=\delta-\frac{\delta_1}{2}$.
Let  $\wh{\C}_{1}:=\wh{\varphi}_{\delta_1}(\C_1)$  and  $\wh{\C}_{2}:=\wh{\varphi}_{-\delta_2}(\C_2)$.
Note that $\delta_1+\delta_2>\delta$.

\begin{claim}\label{infinitestrip}
The intersection $\Omega\cap \wh{\C}_{i}$, $i=1,2$, is an infinite strip.
\end{claim}

\begin{proof}
Given $t\in\mathbb R$, let  $\BH^2_t$ denote $\BH^2\times\{t\}$. Let $\tau^i_t:=\C_i\cap \BH^2_t$
and  $\wh{\tau}^i_t:=\wh{\C}_i\cap \BH^2_t$. Note that for $i=1,2$, $\tau^i_t$ is a circle in
$\BH^2_t$ of radius $b_i(t)$ centered at $(0,0,t)$ while $\wh{\tau}^1_t$ is a circle in $\BH^2_t$
of radius $b_1(t)$ centered at $p_{1,t}:=(\delta_1,0,t)$ and $\wh{\tau}^2_t$ is a circle in $\BH^2_t$
of radius $b_2(t)$ centered at $p_{2,t}:=(-\delta_2,0,t)$. We claim that for any $t\in\mathbb R$,
the intersection $\wh{\tau}^i_t\cap \Omega$ is an arc with end points in $\tau^i_t$, $i=1,2$.
This result would give that $\Omega\cap \wh{\C}_{i}$ is an infinite strip.  We next prove this claim.

Consider the case $i=1$ first. Since $\delta_1<\eta_1\leq b_1(t)$, the center $p_{1,t}$ is inside
the disk in $\BH^2_t$ bounded by $\tau^1_t$. Since the radii of $\tau^1_t$ and $\wh{\tau}^1_t$ are
both equal to $b_1(t)$, then the intersection $\tau^1_t\cap\wh{\tau}^1_t$ is nonempty. It remains to
show that $\wh{\tau}^1_t\cap \tau^2_t=\emptyset$, namely that $b_1(t)+\delta_1<b_2(t)$. This follows because
\[
\delta_1<\delta=\inf_{t\in \mathbb R}(b_2(t)-b_1(t)).
\]
This argument shows that $\Omega\cap \wh{\C}_{1}$ is an infinite strip.

Consider now the case $i=2$. Since $\delta_2<\delta< b_2(t)$, the center $p_{2,t}$ is inside
the disk in $\BH^2_t$ bounded by $\tau^2_t$. Since the radii of $\tau^2_t$ and $\wh{\tau}^2_t$ are
both equal to $b_2(t)$, then the intersection $\tau^2_t\cap\wh{\tau}^2_t$ is nonempty. It remains to
show that $\tau^1_t\cap \wh{\tau}^2_t=\emptyset$, namely that $b_2(t)-\delta_2>b_1(t)$. This follows because
\[
b_2(t)-b_1(t)\geq \inf_{t\in \mathbb R}(b_2(t)-b_1(t))=\delta>\delta_2
\]
This completes the proof  that $\Omega\cap \wh{\C}_{2}$ is an infinite
strip and finishes the proof of the claim.
\end{proof}

\begin{figure}[b]

\relabelbox  {\epsfxsize=3.5in

\centerline{\epsfbox{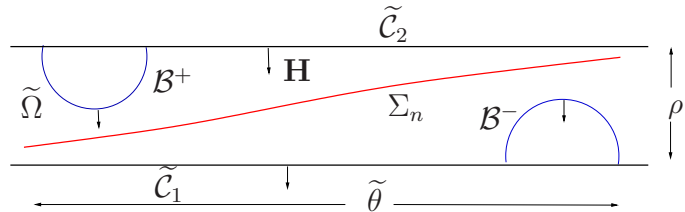}}}

\relabel{1}{$\B^+$}
\relabel{2}{$\wt{\theta}$}
\relabel{3}{$\mathbf{H}$}
\relabel{4}{$\B^-$}
\relabel{5}{$\Sigma_n$}
\relabel{6}{$\wt{\C}_2$}
\relabel{7}{$\wt{\C}_1$}
\relabel{8}{$\wt{\Omega}$}
\relabel{9}{\small $\rho$}

\endrelabelbox

\caption{\label{bumps} \small The position of the bumps $\B^\pm$ in $\wt{\Omega}$
is shown in the picture. The small arrows show the mean curvature vector direction.
The $H$-surfaces $\Sigma_n$ are disjoint from the infinite strips $\B^\pm$ by construction. }
\end{figure}

Now, let $Y^+:=\Omega\cap \wh{\C}_2$ and let $Y^-:=\Omega\cap \wh{\C}_1$. In light of Claim~\ref{infinitestrip}
and its proof, we know  that  $Y^+\cap \C_1=\emptyset$ and  $Y^-\cap \C_2=\emptyset$.

\begin{remark}\label{rmkbumps}
Note  that by construction,
any rotational surface contained in $\Omega$ must intersect
$\wh{\C}_{1}\cup\wh{\C}_{2}$. In particular, $Y^+\cup Y^-$
intersects all $H$-catenoids $\C_d$ for $d\in(d_1,d_2)$ as the circles $\C_d\cap \BH^2_t$ intersect either the
circle $\wh{\tau}^2_t$ or the circle $\wh{\tau}^1_t$ for some $t>0$ since $\delta_1+\delta_2>\delta$.
\end{remark}

In $\wt{\Omega}$, let $\B^+$ be the lift of $Y^+$ in $\wt{\Omega}$ which intersects the slice $\{\wt{\theta}=-10\pi\}$.
Similarly, let $\B^-$ be the lift of $Y^-$ in $\wt{\Omega}$ which intersects the slice $\{\wt{\theta}=10\pi \}$.
Note that each lift of $Y^+$ or $Y^-$is contained in a region where the  $\wt{\theta}$ values of their
points lie in ranges of the form $(\theta_0-\pi,\theta_0+ \pi)$ and so
$\B^+\cap \B^-=\emptyset$. See Figure~\ref{bumps}.

The $H$-surfaces $\B^\pm$ near the top and bottom of $\wt{\Omega}$ will act as barriers (infinite bumps) in
the next section, ensuring that the limit $H$-plane of a certain sequence of compact
$H$-surfaces does not collapse to
an $H$-lamination of
 $\wt{\Omega}$ all of whose leaves are invariant under translations in  the $\wt{\theta}$-direction.

Next we modify $\wt{\Omega}$ as follows. Consider the  component
of $\wt{\Omega}-(\B^+\cup\B^-)$ containing the slice $\{\wt{\theta}=0\}$.
From now on we will call the {\bf closure} of this region $\wt{\Omega}^*$.

\subsection{The Compact Exhaustion of $\wt{\Omega}^*$} \

Consider the rotationally invariant $H$-planes $E_H,-E_H$  described in Section~\ref{prelim}.
Recall that $E_H$ is a graph over the horizontal slice $\BH^2_0$ and it is also tangent to $\BH^2_0$
at the origin. Given $t\in\mathbb R$, let $E_H^t=-E_H+(0,0,t)$ and $-E_H^t=E_H-(0,0,t)$.
Both families $\{E_H^t\}_{t\in\mathbb R}$ and $\{-E_H^t\}_{t\in\mathbb R}$ foliate $\BHH$. Moreover,
there exists $n_0\in\mathbb N$ such that for any $n>n_0$, $n\in\mathbb N$, the following holds.
The highest (lowest) component of the
intersection $S_n^+:=E_H^n\cap \Omega$ ($S_n^-:=-E_H^n\cap \Omega$) is a rotationally
invariant annulus with boundary components
contained in $\C_1$ and $\C_2$. The annulus $S_n^+$ lies ``above'' $S_n^-$ and their intersection is empty.
The region $\U_n$ in $\Omega$ between  $S_n^+$ and $S_n^-$ is a solid torus, see Figure~\ref{univcover2}-left,
and the mean curvature vectors of $S_n^+$ and $S_n^-$ point into $\U_n$.

\begin{figure}[h]
\begin{center}
$\begin{array}{c@{\hspace{.2in}}c}

\relabelbox  {\epsfysize=2.5in \epsfbox{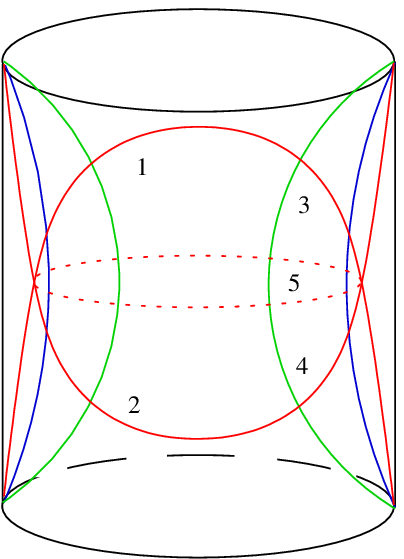}}
\relabel{1}{\footnotesize $E_H^n$}
\relabel{2}{\footnotesize $-E_H^n$}
\relabel{3}{\footnotesize $S_n^+$}
\relabel{4}{\footnotesize $S_n^-$}
\relabel{5}{\footnotesize $\U_n$} \endrelabelbox

&

\relabelbox  {\epsfxsize=2.5in \epsfbox{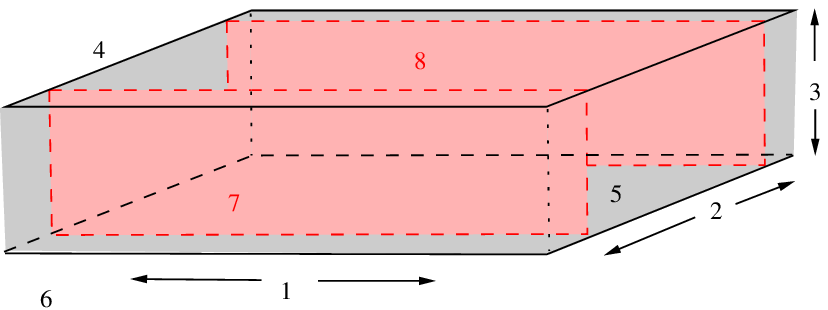}}
\relabel{1}{\footnotesize $\wt{\theta}$}
\relabel{2}{$t$}
\relabel{3}{\footnotesize $\rho$}
\relabel{4}{\footnotesize $\wt{\C}_2$}
\relabel{5}{\footnotesize $\wt{\U}_n$}
\relabel{6}{\footnotesize $\wt{\C}_1$}
\relabel{7}{\footnotesize \color{red} $\wt{S}_n^+$}
\relabel{8}{\footnotesize \color{red} $\wt{S}_n^-$} \endrelabelbox \\ [0.4cm]

\end{array}$

\end{center}

\caption{\label{univcover2} $\U_n = \Omega\cap \wh{\U}_n$ and $\wt{\U}_n$ denotes
its universal cover. Note that $\partial \wt{\U}_n\subset
\wt{\C}_1\cup\wt{\C}_2\cup \wt{S}^+_n\cup \wt{S}^-_n$.}

\end{figure}

Let $\wt{\U}_n\subset \wt{\Omega}$ be the universal cover of $\U_n$, see Figure~\ref{univcover2}-right.
Then, $\partial \wt{\U}_n- \partial \wt{\Omega}= \wt{S}_n^+\cup\wt{S}_n^-$ where can
view $\wt{S}_n^\pm$ as a lift to $\wt{\U}_n$ of the
universal cover of the annulus $S_n^\pm$. Hence, $\wt{S}_n^\pm$ is an infinite $H$-strip in $\wt{\Omega}$,
and the mean curvature vectors of the surfaces $\wt{S}_n^+, \wt{S}_n^-$  point into $\wt{\U}_n$ along $\wt{S}_n^\pm$.
Note that each $\wt{\U}_n$ has bounded $t$-coordinate. Furthermore, we can view $\wt{\U}_n$  as $(\U_n\cap \p_0)\times \BR$,
where $\p_0$ is the half-plane $\{\theta=0\}$ and the second coordinate is $\wt \theta$.
Abusing the notation, we {\bf redefine} $\wt{\U}_n$ to be  $\wt{\U}_n\cap \wt{\Omega}^*$,
that is we have removed the infinite bumps $\B^\pm$ from $\wt{\U}_n$.

Now, we will perform  a sequence of modifications of  $\wt{\U}_n$  so that
for each of these modifications, the $\wt{\theta}$-coordinate
in $\wt{\U}_n$ is  bounded and so that we obtain a compact exhaustion of $\wt{\Omega}^*$.
In order to do this, we will use arguments that are similar to those in Claim~\ref{infinitestrip}.
Recall that the necksize of $\C_2$ is $\eta_2=b_2(0)$. Let $\wh{\C}_3= \wh{\varphi}_{\eta_2}(\C_2)$,
see equation~\eqref{shifting} for the definition of $\wh{\varphi}_{\eta_2}$. Then, $\wh{\C}_3$ is a
rotationally invariant catenoid whose rotational axis is the line $(\eta_2, 0)\times \BR$ (Figure~\ref{strips}-left).
%We claim that $\wh{\C}_3\cap \Omega$ is a pair of infinite strips.

\begin{lemma}
The intersection $\wh{\C}_3\cap \Omega$ is a pair of infinite strips.
\end{lemma}

\begin{proof} It suffices to show that $\wh{\C}_3\cap \C_1$ and $\wh{\C}_3\cap \C_2$ each consists of a pair
of infinite lines. Now, consider the horizontal circles $\tau^1_t, \tau^2_t$, and $\wh{\tau}^3_t$ in the
intersection of $\BH^2_t$ and $\C_1,\C_2$, and $\wh{\C}_3$ respectively, where $\BH^2_t= \BH^2\times\{t\}$.
For any $t\in\BR$, $\tau^i_t$ is a circle of radius $b_i(t)$ in $\BH^2_t$ with center $(0,0,t)$. Similarly,
$\wh{\tau}^3_t$ is a circle of radius $b_2(t)$ in $\BH^2_t$ with center $(\eta_2,0,t)$, see Figure~\ref{strips}-right.
  Hence, it suffices to show that for any $t\in\BR$ each of the
intersection $\tau^1_t\cap\wh{\tau}^3_t$ and $\tau^2_t\cap\wh{\tau}^3_t$ consists of two points.

\begin{figure}[h]
\begin{center}
$\begin{array}{c@{\hspace{.2in}}c}

\relabelbox  {\epsfysize=2in \epsfbox{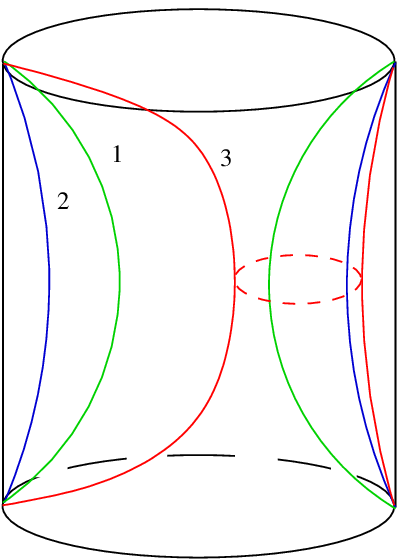}}
\relabel{1}{\footnotesize $\C_1$}
\relabel{2}{\footnotesize $\C_2$}
\relabel{3}{\footnotesize $\wh{\C}_3$}
\endrelabelbox

&

\relabelbox  {\epsfysize=2in \epsfbox{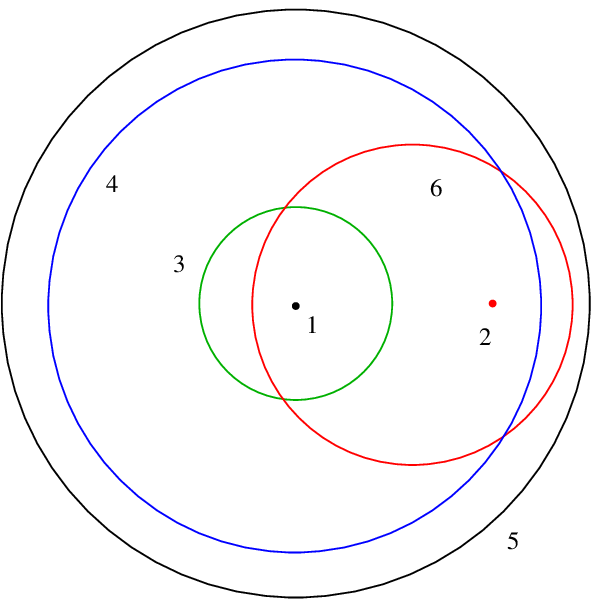}}
\relabel{1}{\footnotesize $O$}
\relabel{2}{$C$}
\relabel{3}{\footnotesize $\tau^1_t$}
\relabel{4}{\footnotesize $\tau^2_t$}
\relabel{5}{\footnotesize $\BH^2_t$}
\relabel{6}{\footnotesize $\wh{\tau}^3_t$}
\endrelabelbox \\ [0.4cm]

\end{array}$

\end{center}

\caption{\label{strips} $\tau^i_t=\C_i\cap \BH^2_t$ is a round circle of
radius $b_i(t)$ with center $O$. $\wh{\tau}^3_t=\wh{\C}_3\cap \BH^2_t$ is a round circle of radius $b_2(t)$ with center $C=(\eta_2,0,t)$.}

\end{figure}

By construction, it is easy to see $\tau^2_t\cap\wh{\tau}^3_t$ consists of two points. This is
because $\tau^2_t$ and $\wh{\tau}^3_t$ have the same radius, $b_2(t)$ and $\eta_2+b_2(t)>b_2(t)$
and $\eta_2-b_2(t)>-b_2(t)$. Therefore, it remains to show that $\tau^1_t\cap\wh{\tau}^3_t$  consists of
two points. By construction, this would be the case if $\eta_2-b_2(t)<b_1(t)$ and $\eta_2-b_2(t)>-b_1(t)$.
The first inequality follows because $\eta_2=\inf_{t\in\BR}b_2(t)$.  The second inequality follows from
Lemma~\ref{disjointlem} because
$$
\eta_2>\eta_2-\eta_1=\sup_{t\in\BR}(b_2(t)-b_1(t)).
$$ \par \vspace{-.2cm}
\end{proof}

Now, let $\wh{\C}_3\cap \Omega = T^+\cup T^-$, where $T^+$ is the infinite strip with $\theta\in(0,\pi)$,
and $T^-$ is the infinite strip with $\theta\in(-\pi,0)$. Note that $T^\pm$ is a $\theta$-graph over the
infinite strip $\wh{\p}_0=\Omega\cap \p_0$ where $\p_0$ is the half plane $\{\theta=0\}$. Let $\V$ be the component of
$\Omega-\wh{C}_3$ containing $\wh{\p}_0$. Notice that the mean curvature vector $\HH$ of $\partial \V$ points into $\V$
on both $T^+$ and $T^-$.

Consider the lifts of $T^+$ and $T^-$ in $\wt{\Omega}$. For $n\in\BZ$, let $\wt{T}^+_n$ be the lift of $T^+$
which belongs to the region $\wt{\theta}\in (2n\pi, (2n+1)\pi)$. Similarly, let $\wt{T}^-_n$ be the lift
of $T^-$ which belongs to the region $\wt{\theta}\in ((2n-1)\pi, 2n\pi)$. Let $\V_n$ be the closed region
in $\wt{\Omega}$ between the infinite strips $\wt{T}^-_{-n}$ and $\wt{T}^+_n$.  Notice that for $n$
sufficiently large, $\B^\pm\subset \V_n$.

Next we define the compact exhaustion $\Delta_n$ of $\wt{\Omega}^*$ as follows:  $\Delta_n:=\wt{\U}_n\cap \V_n$.
Furthermore, the absolute value of the mean curvature of $\partial \Delta_n$ is equal to $H$ and the
mean curvature vector $\HH$ of $\partial \Delta_n$ points into $\Delta_n$
on $\partial \Delta_n-[(\partial \Delta_n\cap \wt{\C}_1)\cup \B^-]$.

\subsection{The sequence of $H$-surfaces}

We next define a sequence of compact $H$-surfaces $\{\Sigma_n\}_{n\in \N}$ where $\Sigma_n\subset \Delta_n$.
For each $n$ sufficiently large, we define a simple closed curve $\Gamma_n$ in $\partial \Delta_n$, and then
we solve the $H$-Plateau problem for $\Gamma_n$ in $\Delta_n$. This will provide an embedded $H$-surface
$\Sigma_n$ in $\Delta_n$ with $\partial \Sigma_n=\Gamma_n$ for each $n$.

\vspace{.2cm}

\noindent {\em The Construction of $\Gamma_n$ in $\partial \Delta_n\colon$}

\vspace{.2cm}

First, consider the annulus $\A_n=\partial \Delta_n-(\wt{\C}_1\cup\wt{\C}_2\cup\B^+\cup\B^-)$ in
$\partial \Delta_n$. Let $\wh{l}_n^+= \wt{\C}_1\cap \wt{T}^+_n$, and $\wh{l}_n^-= \wt{\C}_2\cap \wt{T}^-_{-n}$
be the pair of infinite lines in $\wt{\Omega}$.
Let $l^\pm_n=\wh{l}^\pm_n\cap \A_n$. Let $\mu_n^+$ be an
 arc in $\wt{S}^+_n\cap \A_n$, whose $\wt{\theta}$ and $\rho$ coordinates are strictly increasing as a function of the
  parameter and whose endpoints are $l^+_n\cap \wt{S}^+_n$ and $l^-_n\cap \wt{S}^+_n$ (Figure~\ref{gamma_n}-left).
Similarly, define $\mu_n^-$ to be a monotone arc in $\wt{S}^-_n\cap \A_n$ whose endpoints are $l^+_n\cap \wt{S}^-_n$
and $l^-_n\cap \wt{S}^-_n$.
Note that these arcs
$\mu^+_n$ and $\mu^-_n$ are by construction disjoint from the infinite
bumps $\B^\pm$. Then, $\Gamma_n=\mu^+_n\cup l^+_n \cup \mu^-_n\cup l^-_n$ is a simple closed curve
in $\A_n\subset \partial \Delta_n$ (Figure~\ref{gamma_n}-right).

\begin{figure}[h]
\begin{center}
$\begin{array}{c@{\hspace{.2in}}c}

\relabelbox  {\epsfxsize=2.5in \epsfbox{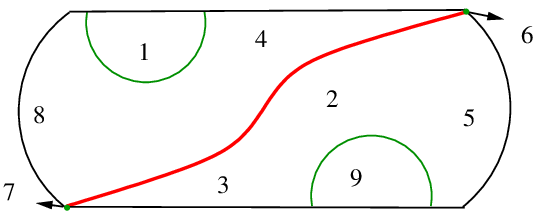}}
\relabel{1}{\footnotesize $\B^+$}
\relabel{2}{\footnotesize $\mu_n^+$}
\relabel{3}{\footnotesize $\wt{\C}_2$}
\relabel{4}{\footnotesize $\wt{\C}_1$}
\relabel{5}{\footnotesize $\wt{T}^+_n$}
\relabel{6}{\footnotesize $l^+_n$}
\relabel{7}{\footnotesize $l^-_n$}
\relabel{8}{\footnotesize $\wt{T}^-_{-n}$}
\relabel{9}{\footnotesize $\B^-$}
\endrelabelbox &

\relabelbox  {\epsfxsize=2.5in \epsfbox{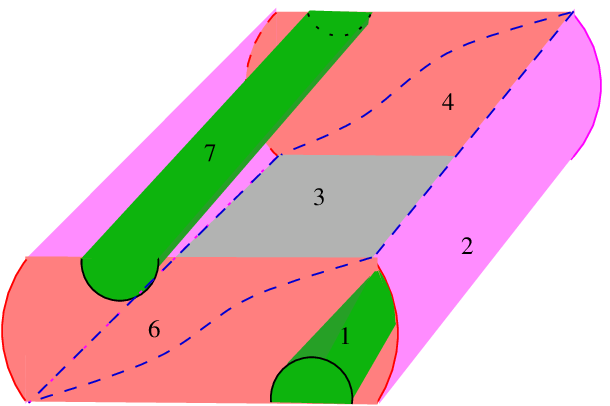}}
\relabel{1}{\footnotesize $\B^-$}
\relabel{2}{\footnotesize $\wt{T}^+_n$}
\relabel{3}{\footnotesize $\Delta_n$}
\relabel{4}{\footnotesize $\wt{S}^-_n$}
\relabel{6}{\footnotesize $\Gamma_n$}
\relabel{7}{\footnotesize $\B^+$}
\endrelabelbox \\ [0.4cm]

\end{array}$

\end{center}

\caption{\label{gamma_n} In the left, $\mu^n_+$  is
pictured in $\wt{S}^+_n$. On the right, the curve $\Gamma_n$  is described in
$\partial \Delta_n$.}

\end{figure}

Next, consider the following variational problem ($H$-Plateau problem):
Given the simple closed curve $\Gamma_n$ in $\A_n$, let $M$ be
a smooth compact embedded surface in $\Delta_n$
with $\partial M=\Gamma_n$. Since $\Delta_n$ is simply-connected, $M$ separates
$\Delta_n$ into two regions. Let $Q$ be the region in $\Delta_n-\Sigma$
with $Q\cap \wt{\C}_2\neq \emptyset$, the ``upper'' region.
Then define the functional $\I_H=\mbox{\rm Area}(M)+2H\,\mbox{\rm Volume}(Q)$. 

By working with integral currents, it is known that
there  exists a smooth (except at the 4 corners of $\Gamma_n$), compact,
embedded $H$-surface $\Sigma_n\subset \Delta_n$ with
$\text{Int}(\Sigma_n)\subset \text{Int}(\Delta_n)$ and
$\partial\Sigma_n=\Gamma_n$. Note that in our setting, $\Delta_n$ is not $H$-mean convex along $\Delta_n\cap \wt{\C}_1$. However, the mean curvature vector along $\Sigma_n$ points outside $Q$ because of the construction of the variational problem. Therefore $\Delta_n\cap \wt{\C}_1$ is still a good barrier for solving the $H$-Plateau problem. 
In fact,  $\Sigma_n$ can be chosen to be,
and we will assume it is, a minimizer for this variational
problem, i.e.,
$I(\Sigma_n)\leq I(M)$ for any
$M\subset \Delta_n$ with $\partial M =\Gamma_n$;
see for instance~\cite[Theorem 2.1]{ton1} and~\cite[Theorem 1]{alr1}.
In particular, the fact that $\text{Int}(\Sigma_n)\subset \text{Int}(\Delta_n)$
is proven in Lemma 3 of~\cite{gu2}.
Moreover,  $\Sigma_n$ separates $\Delta_n$ into two
regions.

Similarly to Lemma~4.1 in~\cite{cmt1}, in the following lemma we show that for any
such $\Gamma_n$, the minimizer surface $\Sigma_n$ is a  $\wt{\theta}$-graph.

\begin{lemma} \label{translation} Let $E_n:=\A_n\cap \wt{T}^+_n$. The minimizer surface  $\Sigma_n$ is a
 $\wt{\theta}$-graph over the compact disk $E_n$. In particular, the related
 Jacobi function $J_n$ on $\Sigma_n$ induced by the inner product of
 the unit normal field to $\Sigma_n$ with the Killing field
 $ \partial_{\wt{\theta}}$
 is positive in the interior of $\Sigma_n$.
\end{lemma}

\begin{proof} The proof is almost identical to the proof of Lemma~4.1 in~\cite{cmt1},
and for the sake of completeness, we give it here. Let  $T_\alpha$ be the isometry
of $\wt{\Omega}$ which is a translation by $\alpha$ in the $\wt{\theta}$ direction, i.e.,
\begin{equation}\label{transl}
T_{\alpha}(\rho, \wt{\theta}, t)=(\rho, \wt{\theta}+\alpha, t).
\end{equation}
Let $T_\alpha(\Sigma_n)=\Sigma^\alpha_n$ and
$T_\alpha(\Gamma_n)=\Gamma^\alpha_n$. We claim that
$\Sigma^\alpha_n\cap\Sigma_n=\emptyset $
for any $\alpha \in \mathbb{R}\setminus\{0\}$ which implies that $\Sigma_n$ is a
 $\wt{\theta}$-graph; we will use that   $\Gamma^\alpha_n $ is disjoint from $\Sigma_n$
 for any $\alpha \in \mathbb{R}\setminus\{0\}$.

Arguing by contradiction, suppose that $\Sigma^\alpha_n\cap\Sigma_n\neq\emptyset $  for a certain
$\alpha\neq 0$.   By compactness of $\Sigma_n$,
there exists a largest positive number $\alpha'$ such that
$\Sigma^{\alpha'}_n\cap\Sigma_n\neq \emptyset $.
Let $p\in \Sigma^{\alpha'}_n\cap\Sigma_n$. Since
$\partial \Sigma^{\alpha'}_n \cap \partial \Sigma_n =\emptyset $
and the interior of $\Sigma_n$, respectively
 $\Sigma^{\alpha'}_n$, lie in the interior of $\Delta_n$, respectively $T_{\alpha'}(\Delta_n)$, then
$p\in \Int(\Sigma^{\alpha'}_n) \cap \Int(\Sigma_n)$.
Since the surfaces $\Int(\Sigma^{\alpha'}_n)$,
$\Int(\Sigma_n)$ lie on one side of each other
and intersect tangentially at the point $p$ with
the same mean curvature vector,
then we obtain a contradiction to the mean curvature comparison
principle for constant mean curvature surfaces, see Proposition~\ref{max}. This proves that $\Sigma_n$ is
graphical over its $\wt \theta$-projection to
$E_n$.

Since by construction every integral curve,
$(\overline \rho,s,\overline t)$ with $\overline \rho, \overline t$
fixed and  $(\overline \rho,s_0, \overline t)\in E_n$ for a certain $s_0$,
of the Killing field $\partial _{\wt{\theta}}$ has
non-zero intersection number with any compact
surface  bounded by $\G_n$, we conclude that
every such integral curve intersects both the
disk $E_n$ and $\Sigma_n$ in single points.
This means that   $\Sigma_n$ is a $\wt{\theta}$-graph over
$E_n$ and thus the related Jacobi function
$J_n$ on $\Sigma_n$ induced by the inner product of
 the unit normal field to $\Sigma_n$ with the
 Killing field $\partial _{\wt{\theta}}$
 is non-negative in the interior of $\Sigma_n$. Since $J_n$ is a
 non-negative Jacobi function, then either
 $J_n\equiv 0$ or $J_n>0$. Since by construction $J_n$ is positive somewhere
 in the interior, then $J_n$ is positive everywhere in the interior.
 This finishes the proof of the lemma.
\end{proof}

\section{The proof of Theorem~\ref{main}}

With $\Gamma_n$ as previously
described, we have so far constructed a sequence
of  compact stable $H$-disks $\Sigma_n$
with $\partial \Sigma_n = \Gamma_n \subset \partial \Delta_n$.
Let $J_n$ be the related non-negative Jacobi function described in Lemma~\ref{translation}.

By the curvature estimates for stable $H$-surfaces given  in~\cite{rst1},
the norms of the second fundamental forms of the $\Sigma_n$ are uniformly
bounded from above at points which are at intrinsic distance at least one
from their boundaries. Since the boundaries of the
 $\Sigma_{n}$ leave every compact subset of $\wt{\Omega}^*$, for each compact set of $\wt{\Omega}^*$,
 the norms of the
 second fundamental forms of the $\Sigma_n$ are uniformly
bounded for values $n$ sufficiently large and such a bound does not depend on the
chosen compact set. Standard compactness arguments give that, after passing to a
subsequence, $\Sigma_n$ converges to a (weak) $H$-lamination $\wt{\cL}$ of $\wt{\Omega}^*$
and the leaves of $\wt{\cL}$ are complete and have uniformly bounded norm of their
second fundamental forms, see for instance~\cite{mr13}.

Let $\beta$ be a compact embedded arc contained in $\wt{\Omega}^*$ such that its
end points $p_+$ and $p_-$ are contained respectively in $\B^+$ and $\B^-$, and
such that these are the only points in the intersection $[\B^+\cup\B^-]\cap \beta$.
Then, for $n$-sufficiently large, the linking number between $\Gamma_n$ and $\beta$
is one, which gives that, for $n$ sufficiently large, $\Sigma_n$ intersects $\beta$
in an odd number of points. In particular $\Sigma_n\cap \beta\neq \emptyset$ which
implies that the lamination $\wt{\cL}$ is not empty.
 \begin{remark}\label{rmkbumps2}
 By  Remark~\ref{rmkbumps}, a leaf of $\wt{\cL}$ that is invariant
 with respect to $\wt{\theta}$-translations cannot be contained in $\wt{\Omega}^*$.
 Therefore none of the leaves of $\wt{\cL}$ are invariant with respect to $\wt{\theta}$-translations.
 \end{remark}

 Let $\wt L$ be a leaf of $\wt{\cL}$ and let $J_{\wt L}$ be the Jacobi function
 induced by taking the
inner product of $\partial_{\wt{\theta}}$ with the unit normal of $\wt L$.  Then,
by the nature of the convergence, $J_{\wt L}\geq 0$ and therefore since it is a Jacobi
field, it is either positive or identically zero. In the latter case, $\wt{\cL}$ would
be invariant with respect to $\wt{\theta}$-translations, contradicting Remark~\ref{rmkbumps2}.
Thus, by Remark~\ref{rmkbumps2}, we have that $J_{\wt L}$ is positive and therefore $\wt L$
is a Killing graph with respect to
$\partial_{\wt{\theta}}$.

\begin{claim}\label{propembleaf}
Each leaf $\wt L$ of $\wt{\cL}$ is properly embedded in $\wt \Omega^*$.
\end{claim}
\begin{proof}
Arguing by contradiction, suppose there exists a leaf  $\wt L$ of $\wt{\cL}$ that is
NOT proper in $\wt \Omega^*$. Then, since the leaf $\wt L$ has uniformly bounded norm
of its second fundamental form, the closure of $\wt L$ in $\wt \Omega^*$ is a lamination
of $\wt \Omega^*$ with a limit leaf $\Lambda$, namely $\Lambda\subset\ov{\wt L}-\wt L$.
Let $J_{\Lambda}$ be the Jacobi function induced by taking the
inner product of $\partial_{\wt{\theta}}$ with the unit normal of $\Lambda$.

Just like in the previous discussion, by the nature of the convergence, $J_{\Lambda}\geq 0$
and therefore, since it is a Jacobi field, it is either positive or identically zero.
In the latter case, $\Lambda$ would be invariant with respect to $\wt{\theta}$-translations
and thus, by Remark~\ref{rmkbumps2}, $\Lambda$ cannot be contained in $\wt \Omega^*$. However,
since $\Lambda$ is contained in the closure of $\wt{ L}$,  this would imply that $\wt L$ is not
contained in  $\wt \Omega^*$, giving a contradiction. Thus,  $J_{\Lambda}$ must be positive
and therefore, $\Lambda$  is a Killing graph with respect to
$\partial_{\wt{\theta}}$. However, this implies that $\wt{L}$ cannot be a Killing graph with respect to
$\partial_{\wt{\theta}}$. This follows because if we fix a point $p$ in $\Lambda$ and
let $U_p\subset \Lambda$ be neighborhood of such point, then by the nature of the convergence,
$U_p$ is the limit of a sequence of disjoint domains $U_{p_n}$ in $\wt{L}$ where $p_n\in \wt{L}$
is a sequence of points converging to $p$ and $U_{p_n}\subset\wt L$ is a neighborhood of $p_n$.
While each domain $U_{p_n}$ is a Killing graph with respect to
$\partial_{\wt{\theta}}$, the convergence to $U_p$ implies that their union is not. This gives a
contradiction and proves that $\Lambda$  cannot be a Killing graph with respect to
$\partial_{\wt{\theta}}$. Since we have already shown that $\Lambda$  must be a Killing graph with respect to
$\partial_{\wt{\theta}}$, this gives a contradiction. Thus $\Lambda$ cannot exist and
each leaf $\wt L$ of $\wt{\cL}$ is properly embedded in $\wt \Omega^*$.
\end{proof}

Arguing similarly to the proof of the previous claim, it follows that a small perturbation of $\beta$, which we still denote by $\beta$ intersects $\Sigma_n$ and
$\wt{\cL}$ transversally in a finite number of points.  Note that  $\wt{\cL}$ is obtained as the limit of
$\Sigma_n$. Indeed, since $\Sigma_n$ separates $\B^+$ and $\B^-$ in $\wt{\Omega}^*$, the algebraic intersection number of $\beta$ and $\Sigma_n$ must be one, which implies that $\beta$ intersects $\Sigma_n$ in an odd number of points. Then $\beta$
intersects $\wt{\cL}$ in an odd number of points and  the claim below follows.

\begin{claim}
	The curve $\beta$ intersects $\wt{\cL}$ in an odd number of points.
\end{claim}

In particular $\beta$ intersects only a finite collection of leaves in $\wt{\cL}$ and we
let $\cF$ denote the non-empty  finite collection of leaves that intersect $\beta$.

\begin{definition}
Let $(\rho_1, \wt\theta_0, t_0)$ be a fixed point in $\wt\C_1$ and let
$\rho_2(\wt\theta_0, t_0)>\rho_1$ such that $(\rho_2(\wt\theta_0, t_0), \wt\theta_0, t_0)$ is in $\wt\C_2$.
Then we call  the arc  in $\wt \Omega$ given by
\begin{equation}\label{alpha}
(\rho_1+s(\rho_2-\rho_1), \wt\theta_0, t_0), \quad s\in[0,1].
\end{equation}
 the vertical line segment based at $(\rho_1, \wt\theta_0, t_0)$.
\end{definition}

\begin{claim} \label{odd}
There exists at least one leaf $\wt{L}_{\beta}$ in $\cF$ that intersects $\beta$ in an odd number
of points and the leaf $\wt{L}_{\beta}$ must intersect each vertical line segment at least once.
\end{claim}

\begin{proof}
The existence of $\wt{L}_{\beta}$ follows because otherwise, if all the leaves in $\cF$
intersected $\beta$ in an even number of points, then  the number of points in the
intersection $\beta \cap \cF$ would be even. Given $\wt{L}_{\beta}$ a leaf in $\cF$
that intersects $\beta$ in an odd number of points, suppose there exists a vertical line
segment which does not intersect $\wt{L}_{\beta}$. Then since by Claim~\ref{propembleaf} $\wt{L}_{\beta}$
is properly embedded, using elementary separation arguments would give that  the number of points
of intersection in $\beta\cap \wt{L}_{\beta}$ must be zero mod 2, that is even, contradicting the previous statement.
 \end{proof}

Let $\Pi$ be the covering map defined in equation~\eqref{covermap} and let $\cP_H:=\Pi(\wt{L}_{\beta})$.
The previous discussion and the fact that $\Pi$ is a local diffeomorphism, implies
that $\cP_H$ is a stable complete $H$-surface embedded in $\Omega$. Indeed, $\cP_H$
is a graph over its $\theta$-projection to $\Int(\Omega)\cap \{(\rho,0,t)\mid
\rho>0, \, t\in \R\}$, which we denote by $\theta(\cP_H)$. Abusing the notation,
let $J_{\cP_H}$ be the Jacobi function induced by taking the
inner product of $\partial_{\theta}$ with the unit normal of $\cP_H$, then $J_{\cP_H}$
is positive. Finally, since the norm of the second fundamental form of $\cP_H$ is uniformly bounded,
standard compactness arguments imply that its closure $\ov\cP_H$ is an $H$-lamination $\cL$
of $\Omega$, see for instance~\cite{mr13}.

\begin{claim}
The closure of $\cP_H$ is an $H$-lamination of $\Omega$ consisting
of itself and two $H$-catenoids  $L_1, L_2\subset \Omega$ that form the limit set of  $\cP_H$.
\end{claim}

\begin{remark}
 Note that these two $H$-catenoids are not necessarily the ones which determine $\partial \Omega$.
\end{remark}

\begin{proof}
Given $(\rho_1, \wt\theta_0, t_0)\in \wt\cC_1$, let $\wt{\gamma}$ be the fixed vertical
line segment in $\wt{\Omega}$ based at $(\rho_1, \wt\theta_0, t_0)$, let $\wt{p}_0$ be a
point in the intersection $\wt{L}_\beta\cap\wt{\gamma}$ (recall that by Claim~\ref{odd} such
intersection is not empty) and let  $p_0=\Pi(\wt{p}_0)\in \Pi(\wt{\gamma})\cap \cP_H$. Then,
by Claim~\ref{odd}, for any $i\in\mathbb N$, the vertical line segment $T_{2\pi i}(\wt{\gamma})$ intersects
$\wt{L}_\beta$ in at least a point $\wt{p}_i$, and $\wt{p}_{i+1}$ is above $\wt{p}_i$, where $T$
is the translation defined in equation~\eqref{transl}. Namely,
$\wt{p}_0=(r_0, \wt{\theta}_0, t_0)$, $\wt{p}_i=(r_i, \wt{\theta}_0+2\pi i, t_0)$
and  $r_i<r_{i+1}<\rho_2(\wt\theta_0, t_0)$. The point $\wt{p_i}\in \wt{L}_\beta $
corresponds to the point $p_i= \Pi(\wt{p}_i)=(r_i, \wt{\theta}_0\, \text{mod}\, 2\pi  , t_0)\in \cP_H$.
Let $r(2):=\lim_{i\to \infty }r_i$ then $r(2)\leq \rho_2(\wt \theta_0,t_0)$  and note that
since $\lim_{i\to\infty}(r_{i+1}-r_i)=0$, then the value of the Jacobi function $J_{\cP_H}$
at $ p_i$ must be going to zero as $i$ goes to infinity. Clearly,
the point $Q:=(r(2), \wt \theta_0 \, \text{mod}\, 2\pi, t_0)\in \Omega$ is
in the closure of $\cP_H$, that is  $\cL$. Let $L_2$ be the leaf of $\cL$ containing $Q$.
By the previous discussion $J_{L_2}(Q)=0$. Since by the nature of the convergence,
either $J_{L_2}$ is positive or $L_2$ is rotational, then $L_2$ is rotational, namely an $H$-catenoid.

Arguing similarly but considering the intersection of $\wt{L}_\beta$ with the vertical
line segments $T_{-2\pi i}(\wt{\gamma})$, $i\in \mathbb N$, one obtains
another $H$-catenoid $L_1$, different from $L_2$, in the lamination $\cL$.
This shows that the closure of $\cP_H$ contains the two $H$-catenoids $L_1$ and $L_2$.

Let $\Omega_g$ be the rotationally invariant, connected region of $\Omega-[L_1\cup L_2]$
whose boundary contains $L_1\cup L_2$. Note that since $\cP_H$ is connected and $L_1\cup L_2$
is contained in its closure, then $\cP_H\subset\Omega_g$. It remains to show that
$\cL=\cP_H\cup L_1\cup L_2$, i.e. $\ov{\cP}_H-\cP_H=L_1\cup L_2$.  If $\ov{\cP}_H-\cP_H\neq L_1\cup L_2$
then there would be another leaf $L_3\in \cL\cap \Omega_g$ and by previous argument, $L_3$
would be an $H$-catenoid. Thus $L_3$ would separate $\Omega_g$ into two regions, contradicting
that fact that $\cP_H$ is connected and $L_1\cup L_2$ are contained in its closure.
This finishes the proof of the claim.
\end{proof}

Note that by the previous claim, $\cP_H$ is properly embedded in $\Omega_g$.

\begin{claim}
The $H$-surface $\cP_H$ is simply-connected and every
integral curve of $\partial_\theta$ that lies in $\Omega_g$ intersects $\cP_H$ in exactly one point.
\end{claim}

\begin{proof}
   Let $D_g:=\Int(\Omega_g)\cap \{(\rho,0,t)\mid
\rho>0, \, t\in \R\}$, then $\cP_H$  is a graph over its $\theta$-projection to $D_g$,
that is $\theta (\cP_H)$. Since $\theta\colon \Omega_g\to D_g$ is a proper submersion
and $\cP_H$ is properly embedded in $\Omega_g$, then $\theta (\cP_H)=D_g$, which implies that every
integral curve of $\partial_\theta$ that lies in $\Omega_g$ intersects $\cP_H$ in exactly
one point. Moreover, since $D_g$ is simply-connected, this gives that $\cP_H$ is also
simply-connected. This finishes the proof of the claim.
\end{proof}

From this claim, it clearly follows that $\Omega_g$ is foliated by  $H$-surfaces,
where the leaves of this foliation are
$L_1$, $L_2$ and  the rotated images $\cP_H ({\theta})$ of
$\cP_H$ around the $t$-axis by angles $\theta\in [0,2\pi)$. The existence of the examples
$\Sigma_H$ in the statement of Theorem~\ref{main} can easily be proven by using $\cP_H$. We set $\Sigma_H=\cP_H$, and $C_i=L_i$ for $i=1,2$.
This finishes the proof of Theorem~\ref{main}.

\section{Appendix: Disjoint $H$-catenoids}

In this section, we will show the existence of disjoint $H$-catenoids in $\BHH$. In particular, we will
prove Lemma~\ref{disjointlem}. Given $H\in (0,\ff)$ and $d\in [-2H,\infty)$,
recall that $\eta_d=\cosh^{-1}(\frac{2dH+\sqrt{1-4H^2+d^2}}{1-4H^2})$ and that $\lambda_d\colon[\eta_d,\infty)\to [0,\infty)$
is the function defined as follows.
\begin{equation}\label{Hcat-eqn}
\lambda_d(\rho)= \int ^{\rho}_{\eta_d} \frac{d+2H\cosh r}{\sqrt{\sinh^2 r - ( d+2H\cosh r)^2}}dr.
\end{equation}

Recall that $\lambda_d(\rho)$ is a monotone increasing function with $\lim_{\rho\to\infty}\lambda_d(\rho)= \infty$
and that  $\lambda'_d(\eta_d)=\infty$ when $d\in (-2H,\infty)$. The $H$-catenoid $\C^H_d$, $ d\in (-2H,\infty)$, is
obtained by rotating a generating curve $\wh{\lambda}_d(\rho)$ about the $t$-axis. The generating curve $\wh{\lambda}_d $
is obtained by doubling the curve $(\rho, 0, \lambda_d(\rho))$, $\rho\in[\eta_d,\infty)$, with its
reflection $(\rho, 0, -\lambda_d(\rho))$, $\rho\in[\eta_d,\infty)$.

Finally, recall that $b_{d}(t):=\lambda_d^{-1}(t)$ for $t\geq 0$, hence $b_d(0)=\eta_{d}$, and that
abusing the notation $b_d(t):=b_d(-t)$ for $t\leq 0$.

\newtheorem*{disjointlem}{Lemma~\ref{disjointlem}}
\begin{disjointlem}
[Disjoint $H$-catenoids] Given $d_1>2$ there exist $d_0>d_1$ and $\delta_0>0$ such that for
any $d_2\in [d_0,\infty)$ and $t>0$ then
\[
\inf_{t\in \mathbb R}( b_{d_2}(t)-b_{d_1}(t))\geq \delta_0.
\]
In particular, the corresponding $H$-catenoids are disjoint, i.e., $\C^H_{d_1}\cap\C^H_{d_2}=\emptyset$.

Moreover, $b_{d_2}(t)-b_{d_1}(t)$ is decreasing for $t>0$ and increasing for $t<0$. In particular,
\[
\sup_{t\in \mathbb R}( b_{d_2}(t)-b_{d_1}(t))=b_{d_2}(0)-b_{d_1}(0)= \eta_{d_2}-\eta_{d_1}.
\]
\end{disjointlem}

\begin{proof}
We begin by introducing the following notations that will be used for the computations in the proof of this lemma,
\[
c:=\cosh r=\frac{e^r+e^{-r}}{2},\, s:=\sinh r=\frac{e^r-e^{-r}}{2}.
\]
 Recall that $c^2-s^2=1$ and $c-s=e^{-r}$. Using these notations,
\begin{equation}
\lambda_d(\rho)= \int ^{\rho}_{\eta_d} \frac{d+2H\cosh r}{\sqrt{\sinh^2 r - ( d+2H\cosh r)^2}}\ dr
\end{equation}
can be rewritten as
\begin{equation}\label{Hcat-eqn}
\lambda_d(\rho)= \int ^{\rho}_{\eta_d} \frac{d+2H(s+e^{-r})}{\sqrt{s^2 - ( d+2Hc)^2}}\ dr=f_d(\rho)+J_d(\rho),
\end{equation}
where
\[
f_d(\rho)=\int ^{\rho}_{\eta_d} \frac{2Hs}{\sqrt{s^2 - ( d+2Hc)^2}} \ dr \ \ \mbox{and} \ \ J_d(\rho)=
\int ^{\rho}_{\eta_d} \frac{d+2He^{-r}}{\sqrt{s^2 - ( d+2Hc)^2}} \ dr
\]

First, by using a series of substitutions, we will get an explicit description of $f_d(\rho)$. Then,
we will show that for $d>2$, $J_d(\rho)$ is bounded independently  of $\rho$ and $d$.

\begin{claim}\label{fdrho}
\begin{equation}
f_d(\rho)=\frac{2H }{\sqrt{1-4H^2}}\cosh^{-1} \left(\frac {(1-4H^2)\cosh \rho- 2dH }{\sqrt{d^2+1-4H^2}} \right).
\end{equation}
\end{claim}

\begin{remark}
After finding $f_d(\rho)$, we  used Wolfram Alpha to compute the derivative of $f_d(\rho)$ and verify our claim. For the sake of completeness, we give a proof.
\end{remark}

\begin{proof}[Proof of Claim~\ref{fdrho}]
The proof is a computation with requires several integrations by substitution. Consider
\[
\int \frac{2Hs}{\sqrt{s^2 - ( d+2Hc)^2}} \ dr
\]
By using the fact that $s^2=c^2-1$ and applying the substitution $\{u=c,du=\frac{dc}{dr}dr=sdr\}$ we obtain that
\[
\int \frac{2Hs}{\sqrt{s^2 - ( d+2Hc)^2}} \ dr=\int  \frac{2H}{\sqrt{u^2-1 - ( d+2Hu)^2}}\ du.
\]
Note that
\[
\begin{split}
&u^2-1 - ( d+2Hu)^2=u^2-1-(d^2+4dHu+4H^2u^2)\\
&=(1-4H^2)u^2-4dHu-d^2-1\\
&=(1-4H^2)(u^2-\frac{4dH}{1-4H^2}u+\frac{4d^2H^2}{(1-4H^2)^2})-\frac{4d^2H^2}{1-4H^2}-d^2-1\\
&=(1-4H^2)[(u-\frac{2dH}{(1-4H^2)})^2-(\frac{4d^2H^2}{(1-4H^2)^2}+\frac{d^2+1}{1-4H^2})]\\
&=(1-4H^2)[(u-\frac{2dH}{(1-4H^2)})^2-(\frac{4d^2H^2+(1-4H^2)(d^2+1)}{(1-4H^2)^2})]\\
&=(1-4H^2)[(u-\frac{2dH}{(1-4H^2)})^2-(\frac{d^2+1-4H^2}{(1-4H^2)^2})].
\end{split}
\]
Therefore, by applying a second substitution, $\{w=u-\frac{2dH}{(1-4H^2)}, dw=du\}$, and letting
$a^2=(\frac{d^2+1-4H^2}{(1-4H^2)^2})$ we get that
\[
\int  \frac{2H}{\sqrt{u^2-1 - ( d+2Hu)^2}}\ du=\int  \frac{2H}{\sqrt{1-4H^2}\sqrt{w^2-a^2}}\ dw
\]

By using the fact that $\sec^2x-1=\tan^2x$ and applying a third substitution, $\{w=a\sec t, dw=a\sec t\tan t dt\}$, we obtain that
\[
\begin{split}
\int  \frac{2Ha\sec t\tan t}{\sqrt{1-4H^2}\sqrt{a^2\sec^2 t-a^2}}&\ dt =\int  \frac{2H\sec t}{\sqrt{1-4H^2}}\ dt\\
&=\frac{2H }{\sqrt{1-4H^2}}\ln|\sec t+\tan t|
\end{split}
\]

Therefore
\[
\begin{split}
\int  \frac{2H}{\sqrt{1-4H^2}\sqrt{w^2-a^2}}&\ dw =\frac{2H }{\sqrt{1-4H^2}}\ln|\frac wa+\sqrt{\frac{w^2}{a^2}-1}|\\
&=\frac{2H }{\sqrt{1-4H^2}}\cosh^{-1}( \frac wa)
\end{split}
\]

Since $w=u-\frac{2dH}{(1-4H^2)}$ then
\[
\begin{split}
\int  \frac{2H}{\sqrt{u^2-1 - ( d+2Hu)^2}}&\ du =\frac{2H }{\sqrt{1-4H^2}}\cosh^{-1}\left ( \frac {u-\frac{2dH}{(1-4H^2)}}{a}\right)\\
&=\frac{2H }{\sqrt{1-4H^2}}\cosh^{-1}\left ( \frac {u-\frac{2dH}{(1-4H^2)}}{\frac{\sqrt{d^2+1-4H^2}}{(1-4H^2)}}\right)\\
&=\frac{2H }{\sqrt{1-4H^2}}\cosh^{-1}\left ( \frac {(1-4H^2)u- 2dH }{\sqrt{d^2+1-4H^2}} \right).
\end{split}
\]

Finally, since $u=\cosh r$
\[
\begin{split}
\int ^{\rho}_{\eta_d} \frac{2Hs}{\sqrt{s^2 -
( d+2Hc)^2}}&=\frac{2H }{\sqrt{1-4H^2}}\cosh^{-1}\left( \frac {(1-4H^2)\cosh r- 2dH }{\sqrt{d^2+1-4H^2}} \right)\bigg|_{\eta_d}^\rho\\
&=
\frac{2H }{\sqrt{1-4H^2}}(\cosh^{-1}\left ( \frac {(1-4H^2)\cosh \rho- 2dH }{\sqrt{d^2+1-4H^2}} \right )\\
&-\cosh^{-1}\left ( \frac {(1-4H^2)\cosh \eta_d- 2dH }{\sqrt{d^2+1-4H^2}} \right))
\end{split}
\]
Recall that $\eta_d=\cosh^{-1} (\frac{2dH+\sqrt{1-4H^2+d^2}}{1-4H^2})$ and thus
\[
\frac {(1-4H^2)\cosh \eta_d- 2dH }{\sqrt{d^2+1-4H^2}}=
\frac {(1-4H^2)(\frac{2dH+\sqrt{1-4H^2+d^2}}{1-4H^2})- 2dH }{\sqrt{d^2+1-4H^2}}=1.
\]
This implies that
\[
f_d(\rho)=\frac{2H }{\sqrt{1-4H^2}}\cosh^{-1} \left(\frac {(1-4H^2)\cosh \rho- 2dH }{\sqrt{d^2+1-4H^2}} \right).
\]
\end{proof}

By Claim~\ref{fdrho} we have that
\[
\begin{split}
f_d(\rho)&=\frac{2H }{\sqrt{1-4H^2}}(\cosh^{-1} \frac {(1-4H^2)\cosh \rho- 2dH }{\sqrt{d^2+1-4H^2}} )\\
 &= \frac{2H }{\sqrt{1-4H^2}}(\rho +\ln \frac {1-4H^2 }{\sqrt{d^2+1-4H^2}}) +g_d(\rho),
\end{split}
\]
where $\lim_{\rho\to\infty}g_d(\rho)=0$.

Recall that
$\lambda_d(\rho)= f_d(\rho)+J_d(\rho)$
where
\[
J_d(\rho)=
\int ^{\rho}_{\eta_d} \frac{d+2He^{-r}}{\sqrt{s^2 - ( d+2Hc)^2}}\ dr
=\int ^{\rho}_{\eta_d} \frac{d+2He^{-r}}{\sqrt{c^2-1 - ( d+2Hc)^2}}\ dr.
\]

\begin{claim}\label{jdrho}
\[
\sup_{d\in(2,\infty),\rho\in (\eta_d,\infty)}J_d(\rho)\leq \pi\sqrt{1-2H}.
\]
\end{claim}

\begin{proof}[Proof of Claim~\ref{jdrho}]
Let
\[
\alpha=\dfrac{2dH+\sqrt{1-4H^2+d^2}}{1-4H^2}\,\text{ and }\,\beta=\dfrac{2dH-\sqrt{1-4H^2+d^2}}{1-4H^2}
\]
 be the roots of $c^2-1 - ( d+2Hc)^2$, i.e.
\[
\begin{split}
c^2-1 - ( d+2Hc)^2&= (1-4H^2)(c^2-\frac{4dH}{1-4H^2}c-\frac{1+d^2}{1-4H^2})\\
&=(1-4H^2)(c-\alpha)(c-\beta).
\end{split}
\]
Note that $\alpha=\cosh{\eta_d}$ and that as $H\in(0,\frac{1}{2})$, $\beta<0<\alpha$. Furthermore, $2He^{-r}<2H<1<d$. Thus we have,
\[
\begin{split}
J_d(\rho)&=\int ^{\rho}_{\eta_d} \frac{d+2He^{-r}}{\sqrt{1-4H^2}\sqrt{(c-\alpha)(c-\beta)}}dr \\
&<  \frac{2d}{\sqrt{1-4H^2}} \int_{\eta_d}^\infty \frac{dr}{\sqrt{(c-\alpha)(c-\beta)}}\\
&< \frac{2d}{\sqrt{1-4H^2}\sqrt{\alpha-\beta}} \int_{\eta_d}^\infty \frac{dr}{\sqrt{c-\alpha}},
\end{split}
\]
where the last inequality holds because for $r>\eta_d$, $\cosh{r}>\alpha$ and thus
$\sqrt{\alpha-\beta}< \sqrt{c-\alpha}$. Notice that
$\alpha-\beta= \frac{2\sqrt{1-4H^2+d^2}}{1-4H^2}>\frac{2d}{1-4H^2}$. Therefore
\[
\frac{2d}{\sqrt{1-4H^2}\sqrt{\alpha-\beta}}<\frac{2d}{\sqrt{1-4H^2}}\frac{\sqrt{1-4H^2}}{\sqrt{2d}}=\sqrt{2d}
\]
and
\[
J_d(\rho)< \sqrt{2d}
 \int_{\eta_d}^\infty \frac{dr}{\sqrt{c-\alpha}}.
\]

Applying the substitution $\{ u=c-\alpha, du=sdr=\sqrt{(u+\alpha)^2-1}dr\}$, we obtain that
\begin{equation}
  \int_{\eta_d}^\infty \frac{dr}{\sqrt{c-\alpha}}=  \int_0^\infty \frac{du}{\sqrt{u}\sqrt{(u+\alpha)^2-1}}
\end{equation}
Let $\omega=\alpha-1$. Note that since $d\geq 1$ then $\alpha>1$ and we have that
$(u+\alpha)^2-1> (u+\omega)^2$ as $u>0$. This gives that
\[
\int_0^\infty \frac{du}{\sqrt{u}\sqrt{(u+\alpha)^2-1}}< \int_0^\infty \frac{du}{\sqrt{u}(u+\omega)}
\]

Applying the substitution $\{ v=\sqrt{u}, dv=\dfrac{du}{2\sqrt{u}}\}$ we get
\[
\int_0^\infty \frac{du}{\sqrt{u}(u+\omega)}= \int_0^\infty \frac{2dv}{v^2+\omega}=
  \ \frac{2}{\sqrt{\omega}}\arctan{\frac{w}{\sqrt{\omega}}}\bigg |_0^\infty  <\dfrac{\pi}{\sqrt{\omega}}
\]
and thus
\[
J_d(\rho)< \sqrt{\frac{2d}{\omega}} \pi.
\]

Note that \[
\begin{split}
\omega=\alpha-1&= \frac{2dH+\sqrt{1-4H^2+d^2}}{1-4H^2}-1\\
&>\frac{(1+2H)d}{1-4H^2}-1=\frac{d}{1-2H}-1.
\end{split}
\]
Since $d>2$, we have $2\omega>\dfrac{d}{1-2H}$ and $\dfrac{d}{\omega}<2(1-2H)$. Then
$
\sqrt{\dfrac{2d}{\omega}}<2\sqrt{1-2H}.
$

Finally, this gives that
\[
J_d(\rho)<2\pi\sqrt{1-2H}
\]
independently on $d>2$ and $\rho>\eta_d$. This finishes the proof of the claim.

\end{proof}

Using Claims~\ref{fdrho} and~\ref{jdrho}, we can now prove the next claim.

\begin{claim}\label{separation}
Given $d_2>d_1>2$ there exists $T\in \mathbb R$ such for any  $t>T$, we have that
\[
\begin{split}
\frac{2H }{\sqrt{1-4H^2}}&(\lambda_{d_2}^{-1}(t)-\lambda_{d_1}^{-1}(t))>
\\
> &\frac 12 \ln \sqrt{\frac {d_2^2+1-4H^2 }{d_1^2+1-4H^2}}-2\pi\sqrt{1-2H}.
\end{split}
\]
\end{claim}
\begin{proof}[Proof of Claim~\ref{separation}]
Recall that $\lambda_d(\rho)=f_d(\rho)+J_d(\rho)$ and that by Claims~\ref{fdrho} and~\ref{jdrho} we have that
\begin{equation}
f_d(\rho)= \frac{2H }{\sqrt{1-4H^2}}(\rho +\ln \frac {1-4H^2 }{\sqrt{d^2+1-4H^2}}) +g_d(\rho),
\end{equation}
where $\lim_{\rho\to\infty}g_d(\rho)=0$,
and that
\begin{equation}
\sup_{d\in(2,\infty),\rho\in (\eta_d,\infty)}J_d(\rho)\leq 2\pi\sqrt{1-2H} .
\end{equation}
Let $\rho_i(t):=\lambda_{d_i}^{-1}(t)$, $i=1,2$. Using this notation, since
$t=\lambda_1(\rho_1(t))=\lambda_2(\rho_2(t))$ we obtain that
\[
\begin{split}
&0=\lambda_2(\rho_2(t))-\lambda_1(\rho_1(t))\\
&=f_{d_2}(\rho_2(t))+J_{d_2}(\rho_2(t)) -f_{d_1}(\rho_1(t))-J_{d_1}(\rho_1(t))\\
&=\frac{2H }{\sqrt{1-4H^2}}(\rho_2(t) +\ln \frac {1-4H^2 }{\sqrt{d_2^2+1-4H^2}}) +g_{d_2}(\rho_2(t))+J_{d_2}(\rho_2(t))\\
&-\frac{2H }{\sqrt{1-4H^2}}(\rho_1(t) -\ln \frac {1-4H^2 }{\sqrt{d_1^2+1-4H^2}}) -g_{d_1}(\rho_1(t))-J_{d_1}(\rho_1(t))
\end{split}
\]

Recall that $\lim_{t\to\infty}\rho_i(t)=\infty$, $i=1,2$, therefore given $\ve>0$ there
exists $T_\ve\in \mathbb R$ such that for any $t>T_\ve$, $|g_{d_i}(\rho_i(t))|\leq\ve$. Taking
\[
4\ve<\ln \sqrt{\frac {d_2^2+1-4H^2 }{d_1^2+1-4H^2}}
\]
 for $t>T_\ve$ we get that
\[
\begin{split}
\frac{2H }{\sqrt{1-4H^2}}&(\rho_2(t)-\rho_1(t))>
\\
>&\ln \sqrt{\frac {d_2^2+1-4H^2 }{d_1^2+1-4H^2}}+J_{d_1}(\rho_1(t))-J_{d_2}(\rho_2(t))-2\ve\\
>&\frac 12 \ln \sqrt{\frac {d_2^2+1-4H^2 }{d_1^2+1-4H^2}}+J_{d_1}(\rho_1(t))-J_{d_2}(\rho_2(t)).
\end{split}
\]
Notice that    $J_{d_1}(\rho_1(t))>0$ and that Claim~\ref{jdrho} gives that
\[
\sup_{\rho\in (\eta_{d_2},\infty)}J_{d_2}(\rho)\leq 2\pi\sqrt{1-2H}.
\]
Therefore
\[
\begin{split}
\frac{2H }{\sqrt{1-4H^2}}&(\rho_2(t)-\rho_1(t))>
\\
> &\frac 12 \ln \sqrt{\frac {d_2^2+1-4H^2 }{d_1^2+1-4H^2}}-2\pi\sqrt{1-2H}.
\end{split}
\]
This finishes the proof of the claim.
\end{proof}

We can now use Claim~\ref{separation} to finish the proof of the lemma. Given $d_1>2$ fix $d_0>d_1$ such that
\[
\begin{split}
\frac{\sqrt{1-4H^2}}{4H} \left (\ln \sqrt{\frac {d_0^2+1-4H^2 }{d_1^2+1-4H^2}}-4\pi\sqrt{1-2H}\right)=1.
\end{split}
\]
Then, by Claim~\ref{separation}, given $d_2\geq d_0$ there exists $T>0$ such that
$\lambda_{d_2}^{-1}(t)-\lambda_{d_1}^{-1}(t)>1$ for any $t>T$. Notice that since for any
$\rho\in(\eta_2,\infty)$, $\lambda'_{d_2}(\rho)>\lambda'_{d_1}(\rho)$,
then there exists at most one $t_0>0$ such that $\lambda_{d_2}^{-1}(t_0)-\lambda_{d_1}^{-1}(t_0)=0$.
Therefore,  since there exists $T>0$ such that
$\lambda_{d_2}^{-1}(t)-\lambda_{d_1}^{-1}(t)>1$ for any $t>T$ and
$\lambda_{d_2}^{-1}(0)-\lambda_{d_1}^{-1}(0)=\eta_{d_2}-\eta_{d_1}>0$, this implies that there
exists a constant $\delta (d_2)>0$ such that for any $t>0$,
\[
\lambda_{d_2}^{-1}(t)-\lambda_{d_1}^{-1}(t)>\delta (d_2).
\]

A priori it could happen that $\lim_{d_2\to\infty}\delta (d_2)=0$.
The fact that  $\lim_{d_2\to\infty}\delta (d_2)>0$ follows easy by noticing that by applying
Claim~\ref{separation} and using  the same arguments as in the previous paragraph there exists
$d_3> d_0$ such that for any $d\geq d_3$ and $t>0$,
\[
\lambda_{d}^{-1}(t)-\lambda_{d_0}^{-1}(t)>0.
\]
Therefore,
for any $d\geq d_3$ and $t>0$,
\[
\lambda_{d}^{-1}(t)-\lambda_{d_1}^{-1}(t)> \lambda_{d_0}^{-1}(t)-\lambda_{d_1}^{-1}(t)>\delta (d_0)
\]
which implies that
\[
\lim_{d_2\to\infty}\delta (d_2)\geq \delta (d_0)>0.
\]

Setting $\delta_0=\inf_{d\in[d_0,\infty)}\delta(d_2)>0$ gives that
\[
\inf_{t\in \mathbb R_{\geq 0}}( \lambda_{d_2}^{-1}(t)-\lambda_{d_1}^{-1}(t))\geq \delta_0.
\]
By definition of $b_d(t)$ then
\[
\inf_{t\in \mathbb R}( b_{d_2}(t)-b_{d_1}(t))=
\inf_{t\in \mathbb R_{\geq 0}}( \lambda_{d_2}^{-1}(t)-\lambda_{d_1}^{-1}(t))\geq \delta_0.
\]
It remains to prove that $b_2(t)-b_1(t)$ is decreasing for $t>0$ and increasing for $t<0$.
By definition of $b_d(t)$, it suffices to show that $b_2(t)-b_1(t)$ is decreasing for $t>0$.
We are going to show $\frac{d}{dt}(b_2(t)-b_1(t))<0$ when $t>0$.

By definition of $b_i$, for $t>0$ we have that $\lambda_i(b_i(t))=t$ and thus $b_i'(t)=\frac{1}{\lambda_i'(b_i(t))}$.
By definition of $\lambda_d(t)$ for $t>0$ the following holds,
\[
b_1'(t)=\frac{1}{\lambda_1'(b_1(t))}>\frac{1}{\lambda_1'(b_2(t))}>\frac{1}{\lambda_2'(b_2(t))}=b_2'(t).
\]
The first inequality is due to the convexity of the function $\lambda_1(t)$ and the second inequality is
due to the fact that $\lambda_1'(\rho)<\lambda_2'(\rho)$ for any $\rho>\eta_2$. This proves that
$\frac{d}{dt}(b_2(t)-b_1(t))=b_2'(t)-b_1'(t)<0$ for $t>0$ and finishes the  proof of the claim.
\end{proof}

Note that if $d$ is sufficiently close to $-2H$ then $\C^H_d$ must be unstable. This follows
because as $d$ approaches $-2H$, the norm of the second fundamental form of $\C^H_d$ becomes
arbitrarily large at points that approach the ``origin'' of $\BHH$ and a simple rescaling
argument gives that a sequence of subdomains of $\C^H_d$ converge to a catenoid, which is
an unstable minimal surface. This observation, together with our previous lemma suggests the
following conjecture.

\vspace{.2cm}

\noindent {\bf Conjecture:} Given $H\in(0,\frac12)$ there exists $d_H>-2H$ such that the
following holds. For any $d>d'>d_H$, $\C^H_{d}\cap\C^H_{d'}=\emptyset$, and the family
$\{\C^H_d \mid d\in [d_H,\infty)\}$ gives a foliation of the closure of the non-simply-connected
component of $\BHH-\C^H_{d_H}$. The $H$-catenoid $\C^H_d$ is unstable if $d\in(-2H, d_H)$ and
stable if $d\in(d_H,\infty)$. The $H$-catenoid $\C^H_{d_H}$ is a stable-unstable catenoid.

\bibliographystyle{plain}
\bibliography{bill}
\end{document}